\documentclass[11pt]{article}
\usepackage{amsfonts,amssymb}
\usepackage[margin=1in]{geometry}
\usepackage[T1]{fontenc}
\usepackage{verbatim,graphicx,diagbox}
\usepackage{amsmath,array,multirow,dsfont,amsthm,thmtools,thm-restate,enumerate,bm}
\usepackage{tikz,capt-of}
\usetikzlibrary{arrows}
\usepackage[hidelinks]{hyperref}
\usepackage{xcolor}
\definecolor{cream}{RGB}{203, 237, 204}

\newtheorem{theorem}{Theorem}[section]
\newtheorem{lemma}[theorem]{Lemma}

\newtheorem{claim}[theorem]{Claim}

\newtheorem{question}[theorem]{Question}
\newtheorem{problem}[theorem]{Problem}

\numberwithin{equation}{section}

\newcommand{\A}{{\mathcal{A}}}

\newcommand{\C}{{\mathcal{C}}}
\newcommand{\F}{{\mathcal{F}}}

\newcommand{\T}{{\mathcal{T}}}
\newcommand{\mH}{{\mathcal{H}}}

\title{Counting cliques with prescribed intersection sizes}
\date{}
\author{Yuhao~Zhao\thanks{Y. Zhao ({\tt zhaoyh21@mail.ustc.edu.cn}) is with the School of Mathematical Sciences, University of Science and Technology of China, Hefei, 230026, Anhui, China.} ~and~Xiande~Zhang
	\thanks{X. Zhang ({\tt drzhangx@ustc.edu.cn}) is with the School of Mathematical Sciences,
		University of Science and Technology of China, Hefei, 230026, and with Hefei National Laboratory, University of Science and Technology of China, Hefei 230088, China.
		The research of X. Zhang is supported by the National Key Research and
		Development Programs of China 2023YFA1010200 and 2020YFA0713100, the NSFC
		under Grants No. 12171452 and No. 12231014, and the Innovation Program for Quantum Science and Technology 2021ZD0302902.	}
}
\begin{document}
	\maketitle
	\begin{abstract}
		We study the generalized Tur\'an problem regarding cliques with restricted intersections, which highlights the motivation from extremal set theory. Let $L=\{\ell_1,\dots,\ell_s\}\subset [0,r-1]$ be a fixed integer set with $|L|\notin \{1,r\}$ and $\ell_1<\dots<\ell_s$, and let $\Psi_r(n,L)$ denote the maximum number of $r$-cliques in an $n$-vertex graph whose $r$-cliques are $L$-intersecting as a family of $r$-subsets. Helliar and Liu recently initiated the systematic study of the function $\Psi_r(n,L)$ and showed that $\Psi_r(n,L)\le \left(1-\frac{1}{3r}\right) \prod_{\ell\in L}\frac{n-\ell}{r-\ell}$ for large $n$, improving the trivial bound from the Deza--Erd\H{o}s--Frankl theorem by a factor of $1-\frac{1}{3r}$.  In this article, we improve their result by showing that as $n$ goes to infinity $\Psi_r(n,L)=\Theta_{r,L}(n^{|L|})$ if and only if $\ell_1,\dots,\ell_s,r$ form an arithmetic progression and fully determining the corresponding exact values of $\Psi_r(n,L)$ for sufficiently large $n$ in this case. Moreover, when $L=[t,r-1]$, for the generalized Tur\'an extension of the Erd\H{o}s--Ko--Rado theorem given by Helliar and Liu, we show a Hilton--Milner-type stability result.
	\end{abstract}
	
 \section{Introduction}	
 The generalized Tur\'{a}n problem concerns the maximum number of copies of $F$ in an $n$-vertex $\mathcal{G}$-free graph, where $F$ is a graph and $\mathcal{G}$ is a collection of graphs. When $F$ is an edge, it is exactly the classical Tur\'{a}n problem which was pioneered by Tur\'an~\cite{turan1941external}. The next natural case that $(F,\mathcal{G})=(K_{r},\{K_{\ell}\})$, where $K_r$ denotes the complete graph on $r$ vertices, was considered by Erd\H{o}s~\cite{erdos1962} in 1962. More generally, the systematic study of the generalized Tur\'{a}n problem was initiated by Alon and Shikhelman~\cite{alon2016many}, who obtained many results for various pairs of $(F,\mathcal{G})$.
 In this paper, following a recent work of Helliar and Liu~\cite{helliar2024generalized}, we consider the generalized Tur\'{a}n problem motivated by theorems from extremal set theory.

  Let us begin with some basic definitions. Denote by $[n]$ the integer set $\{1,2,\dots,n\}$. For any two integers $a<b$, let $[a,b]=\{a,a+1,\dots,b\}$. Write $2^{X}$ for the power set of a finite set $X$, and write $\binom{X}{r}$ for the set of all $r$-subsets of $X$. Subsets of $\binom{X}{r}$ are called {\it $r$-uniform hypergraphs} or {\it $r$-graphs} for short.
  Let $L$ be a set of non-negative integers. A family $\F\subset 2^{X}$ is said to be {\it $L$-intersecting} if the intersection size of any two distinct sets in $\F$ belongs to $L$.
  When $L=[t,r-1]$ for some $t\in [r-1]$, an $L$-intersecting $r$-graph $\F\subset \binom{X}{r}$ is simply said to be {\it $t$-intersecting}. For an integer set $L\subset [0,r-1]$, let $\Phi_r(n,L)$ denote the maximum possible size of an $n$-vertex $L$-intersecting $r$-graph, which has been extensively studied in extremal set theory (see, e.g., \cite{frankl2016invitation}).
  The following celebrated Deza--Erd\H{o}s--Frankl theorem provides a general upper bound on $\Phi_r(n,L)$.

  \begin{theorem}[Deza--Erd\H{o}s--Frankl~\cite{deza1978intersection}]\label{thm:DEF78}
  	Let $r \ge 3$, $n \ge 2^r r^3$ and $L \subset [0,r-1]$.
  	If $\mathcal{H}$ is an $n$-vertex $L$-intersecting $r$-graph, then  $|\mathcal{H}|\leq \prod_{\ell\in L}\frac{n-\ell}{r-\ell}$.
  	Moreover, there exists a constant $C=C(r,L)$ such that every $n$-vertex $L$-intersecting $r$-graph $\mathcal{H}$ with $|\mathcal{H}|\ge C n^{|L|-1}$ satisfies $|\bigcap_{A\in\mathcal{H}} A|\ge \min L$.
  \end{theorem}

 The Deza--Erd\H{o}s--Frankl bound is tight for several cases. For example, when $L=[0,t]$, the upper bound can be achieved by combinatorial designs~\cite{Keevash}. When $L=[t,r-1]$, it coincides with the Erd\H{o}s--Ko--Rado theorem\footnote{We also noted that the Deza--Erd\H{o}s--Frankl theorem already implies a recent conjecture of Barg, Glazyrin, Kao, Lai, Tseng and Yu~\cite[Conjecture 1.3]{Barg2024}, which is an extension of the Erd\H{o}s--Ko--Rado theorem.}~\cite{erdos1961intersection}, which states that for large $n$ every $t$-intersecting $r$-graph $\F\subset \binom{[n]}{r}$ has size at most $\binom{n-t}{r-t}$ and equality holds if and only if $\F=\{A\in \binom{[n]}{r}:T\subset A \}$ for some $T\in \binom{[n]}{t}$.
  However, in general it is very hard to find optimal constructions.

  Motivated by the study of $L$-intersecting $r$-graphs, Helliar and Liu~\cite{helliar2024generalized} recently introduced an interesting  generalized Tur\'an problem, which concerns the maximum size of an $L$-intersecting $r$-graph consisting of all $r$-cliques in an $n$-vertex graph.
  For $r\geq 3$, given a graph $G$, define its {\it associated $r$-graph} to be $
  \mH_G^r:= \{ S\in \binom{V(G)}{r} : G[S]\cong K_r\}
  $,
  where $V(G)$ denotes the vertex set of $G$ and $G[S]$ stands for the induced subgraph. For an integer set $L\subset [0, r-1]$, a graph $G$ is called {\it $(K_r, L)$-intersecting} if its associated $r$-graph $\mH_G^r$ is $L$-intersecting. We simply say $G$ is {\it $(K_r,t)$-intersecting} if $G$ is $(K_r,[t,r-1])$-intersecting, i.e., $\mH_G^r$ is $t$-intersecting.  Denote by  $N(K_r, G)$ the number of copies of $K_r$ in $G$. Helliar and Liu~\cite{helliar2024generalized} investigated the extremal function:
  \[
  \Psi_r(n,L):= \max\{ N(K_r, G): G \text{ is an $n$-vertex $(K_r, L)$-intersecting graph}\},
  \]
  which can be rephrased as the generalized Tur\'an number that counts the maximum number of $r$-cliques in an $n$-vertex graph without two $r$-cliques whose intersection has size in $[0,r-1]\setminus L$.

  It is clear that $\Psi_r(n,L)\le \Phi_r(n,L)$ by definition. Thus, Theorem~\ref{thm:DEF78} straightforwardly implies that for sufficiently large $n$ we have
    \begin{equation}\label{eq-from-DEF}
 	\Psi_r(n,L)\leq\prod_{\ell\in L}\frac{n-\ell}{r-\ell}=O(n^{|L|}).
    \end{equation}
  In the particular case of $L=[0,t]$ with $t\in [r-2]$, the well-known Ruzsa--Szemer\'edi $(6,3)$ theorem~\cite{ruzsa1978triple} is in fact equivalent to the statement that $n^{2-o(1)}<\Psi_3(n,\{0,1\})=o(n^2)$,
  which was extended to any fixed $t\in [r-2]$ recently by Gowers and Janzer~\cite{gowers2021generalizations} by showing that
   \begin{equation}\label{eq-Gowers-Janzer}
   n^{t+1-o(1)}<\Psi_r(n,[0,t])=o(n^{t+1})	
   \end{equation}
    as $n$ goes to infinity. Moreover, in another special case that $L = [0,r-1]\setminus \{s\}$, the values of $\Psi_r(n,L)$ have been considered in \cite{liu2021generalized,sos1976remarks}.
    Very recently, Helliar and Liu~\cite{helliar2024generalized} obtained the following general upper bound on $\Psi_r(n,L)$, which indicates that the bound in \eqref{eq-from-DEF} from the Deza--Erd\H{o}s--Frankl theorem can be improved by a constant factor (in terms of $r$) when $2\leq |L|\leq r-1$.

 \begin{theorem}[\cite{helliar2024generalized}]\label{thm:Helliar-Liu-DEF-Turan}
 Let $r \ge 3$ and let $L \subset [0,r-1]$ be a subset with  $2\leq |L|\leq r-1$.
 Then for $n \ge (2r)^{r+1}$,
 	\begin{align}\label{eq-DEFTuran0}
 	\Psi_r(n,L)	 \le \left(1-\frac{1}{3r}\right) \prod_{\ell\in L}\frac{n-\ell}{r-\ell}.
 	 	\end{align}
 \end{theorem}
    As remarked in \cite{helliar2024generalized}, it is not difficult to see that the upper bound in \eqref{eq-DEFTuran0} does not hold for $|L|\in\{1,r\}$. Indeed, the complete graph $K_{n}$ gives a trivial example for $|L|=r$.
 	Moreover, write $G + H$ for the graph obtained from the vertex-disjoint union of graphs $G$ and $H$ by adding all edges between vertex sets $V(G)$ and $V(H)$. For $|L|=1$, one can consider the graph $K_{\ell} + \lfloor\frac{n-\ell}{r-\ell}\rfloor K_{r-\ell}$,
 	where $\lfloor\frac{n-\ell}{r-\ell}\rfloor K_{r-\ell}$ denotes the graph consisting of $\lfloor\frac{n-\ell}{r-\ell}\rfloor$ vertex-disjoint copies of $K_{r-\ell}$.

 	Our first result provides an improvement for Theorem~\ref{thm:Helliar-Liu-DEF-Turan}, by completely characterizing all parameters $(r,L)$ with $\Psi_r(n,L)=\Theta_{r,L}(n^{|L|})$ and determining the corresponding exact values of $\Psi_r(n,L)$ for large $n$, as stated below. We remark that the second part of Theorem~\ref{thm:DEF-Turan} in fact gives a generalized Tur\'an extension of the \textit{uniform eventown} problem considered in~\cite{frankl2016uniform}. Here and throughout, the {\it Tur\'an graph} $T(n,t)$ stands for the balanced complete $t$-partite graph on $n$ vertices.
 	\begin{theorem}\label{thm:DEF-Turan}
 		Let $r \ge 3$ be a fixed integer and let $L = \{\ell_1, \ldots, \ell_{s}\}\subset [0,r-1]$ be a fixed subset of size $s\notin\{1,r\}$ with $\ell_1<\ell_2<\dots<\ell_s$. Then the following statements hold.
 	\begin{enumerate}[(1)]
 	   	\item If $\ell_1,\ell_2,\dots,\ell_s,r$ do not form an arithmetic progression, then $\Psi_r(n,L) = o(n^s) $ as $n\to\infty$. If further $r-\ell_s=\ell_s-\ell_{s-1}$, then $\Psi_r(n,L) = O(n^{s-1})$.
 		\item If $\ell_1,\ell_2,\dots,\ell_s,r$ form an arithmetic progression and let $d:=r-\ell_s=\dots=\ell_2-\ell_1=(r-\ell_1)/s$ be the common difference, then there exists $n_{\ref{thm:DEF-Turan}}=n_{\ref{thm:DEF-Turan}}(r,L)$ such that for any $n>n_{\ref{thm:DEF-Turan}}$ we have
 		\begin{align}\label{eq-AP}
 			\Psi_r(n,L) = N(K_s, T(\lfloor(n-\ell_1)/d\rfloor,s)) = (1+o(1))\left(\frac{n-\ell_1}{r-\ell_1} \right)^{s},
 		\end{align}
 		and the extremal graph is unique up to isomorphism (see the graph $G_{n,r,L}$ in (\ref{gnrl})).
 	\end{enumerate}
 	\end{theorem}
	
	  We mention that, in the first case of \autoref{thm:DEF-Turan}, when $\ell_1,\ell_2,\dots,\ell_s,r$ do not form an arithmetic progression and $\ell_s-\ell_{s-1}=r-\ell_s$, we indeed have $\Phi_r(n,L) = O(n^{s-1})$ for any $L$-intersecting $r$-graphs by our proof. When $L=[t,r-1]$, Helliar and Liu~\cite{helliar2024generalized} proved the following generalized Tur\'an extension of the Erd\H{o}s--Ko--Rado theorem, which is optimal and better than the bound given by \eqref{eq-DEFTuran0}. Note that \eqref{eq-AP} in fact provides an extension of this result from $d=1$ to general $d$.
 	
 	\begin{theorem}[\cite{helliar2024generalized}]\label{thm:Helliar-Liu-EKR-Turan}
 	  Let $r > t \geq 1$ be integers. There exists $n_{\ref{thm:Helliar-Liu-EKR-Turan}}=n_{\ref{thm:Helliar-Liu-EKR-Turan}}(r, t)$ such that every $(K_r, t)$-intersecting graph $G$ on $n \geq  n_{\ref{thm:Helliar-Liu-EKR-Turan}}$ vertices satisfies
 	  \[
 	  N(K_r, G) \leq
 	  N(K_{r-t},  T(n-t, r-t))=(1+o(1))\left(\frac{n-t}{r-t} \right)^{r-t},
 	  \]
 	  and equality holds if and only if $G$ is isomorphic to $K_t + T(n-t, r-t)$. 	
 	\end{theorem}
 	Helliar and Liu~\cite{helliar2024generalized} also gave a  stability result for Theorem~\ref{thm:Helliar-Liu-EKR-Turan}. They proved that there exists a constant $\epsilon_0>0$ such that, for all $\epsilon<\epsilon_0$ and large enough $n$ in terms of $r$ and $t$, every $n$-vertex $(K_r,t)$-intersecting
 	graph $G$ with $N(K_r, G) \ge (1-\epsilon)(\frac{n-t}{r-t})^{r-t}$ satisfies that $G-T$ can be made $(r-t)$-partite by removing at most $2\epsilon(n-t)^2$ edges for some $t$-subset $T\subset V(G)$.
 	
 	On the other hand, for intersecting uniform set systems, the classical Hilton--Milner theorem~\cite{hilton1967some} gives a strong stability result for the Erd\H{o}s--Ko--Rado theorem. An $r$-graph $\mathcal{F}$ is said to be {\it non-trivial $t$-intersecting} if $\mathcal{F}$ is $t$-intersecting and $|\bigcap_{A\in\mathcal{F}}A|< t$. The Hilton--Milner theorem determines the extremal non-trivial $1$-intersecting families, showing that every non-trivial $1$-intersecting $r$-graph $\F\subset \binom{[n]}{r}$ has size at most	$\binom{n-1}{r-1}-\binom{n-r-1}{r-1}+1$ whenever $n>2r$. Later, Frankl~\cite{frankl1978intersecting} further considered non-trivial $t$-intersecting families and characterized all extremal non-trivial $t$-intersecting families $\F\subset\binom{[n]}{r}$ for general $t$ and large enough $n$:
\begin{enumerate}[i)]
 	\item $r>2t+1$ or $r=3$, $t=1$. There exist $S_1\in\binom{[n]}{t},S_2\in\binom{[n]}{r-t+1}$ with $S_1\cap S_2=\emptyset$ such that
 	\begin{equation}\label{eq:1st-extremal}
 		\F= \left\lbrace A\in \binom{[n]}{r} : S_1\subset A,\ A\cap S_2\neq\emptyset \right\rbrace
 		\cup  \left\lbrace A\in \binom{[n]}{r} : S_2\subset A,\ |A\cap S_1|\ge t-1 \right\rbrace.
 	\end{equation}
 	\item $r\le 2t+1$. There exists $S\in\binom{[n]}{t+2}$ such that $
 	\F= \{ A\in \binom{[n]}{r} : |A\cap S|\ge t+1 \}.
 	$
\end{enumerate}
 	\noindent For small $n$, a complete result for the extremal non-trivial $t$-intersecting families was proved by Ahlswede and Khachatrian~\cite{ahlswede1996complete}. Since every edge in the associated $r$-graph of the extremal graph in Theorem~\ref{thm:Helliar-Liu-EKR-Turan} contains a fixed $t$-set, it is natural to give a stability result for Theorem~\ref{thm:Helliar-Liu-EKR-Turan} in the spirit of the Hilton--Milner theorem and determine the extremal graph with a non-trivial $t$-intersecting associated $r$-graph.
 	
 	Our second main result gives a Hilton--Milner-type stability result for Theorem~\ref{thm:Helliar-Liu-EKR-Turan}.
 	We say a graph $G$ is {\it non-trivial $(K_r,t)$-intersecting} if its associated $r$-graph $\mH_G^r$ is non-trivial $t$-intersecting.
 	\begin{theorem}\label{thm:stab}
 		Let $r > t \geq 1$ be integers. There exists $n_{\ref{thm:stab}}=n_{\ref{thm:stab}}(r,t)$ such that for all $n>n_{\ref{thm:stab}}$ the following holds.
 		If an $n$-vertex graph $G$ is non-trivial $(K_r,t)$-intersecting, then
 		\[
 		N(K_r, G) \le N(K_r,  K_{t+2} + T(n-t-2, r-t-1)),
 		\]
 		and equality holds if and only if $G$ is isomorphic to $K_{t+2} + T(n-t-2, r-t-1)$.
 	\end{theorem}
 	
 	 In other words, Theorem~\ref{thm:stab} says that for any large $n$, if $G$ is an $n$-vertex $(K_r,t)$-intersecting graph with $N(K_r, G) > N(K_r,  K_{t+2} + T(n-t-2, r-t-1))$, then $G$ must have the structure that every $r$-clique in $G$ contains a fixed $t$-subset of $V(G)$.
 	 Note that the extremal structure in Theorem~\ref{thm:stab} is similar to the one in ii), in the sense that each $r$-clique has at least $t+1$ vertices in common with a fixed set of cardinality $t+2$.
 	
 	 \paragraph{Organization of the paper.}
 	  We first prove the crucial case $L=\{0,\ell\}$ with $r\neq 2\ell$ of Theorem~\ref{thm:DEF-Turan} (1) in Section~\ref{subsec:special}, and then complete the proof of Theorem~\ref{thm:DEF-Turan} (1) in Section~\ref{subsec:general}. Next, the part (2) of Theorem~\ref{thm:DEF-Turan} is proved in Section~\ref{subsec:last-case}. In Section~\ref{sec:hm}, we prove Theorem~\ref{thm:stab}. Lastly, some final remarks and open problems can be found in Section~\ref{sec:conclusion}.

\section{The maximum number of $r$-cliques in a $(K_r,L)$-intersecting graph}
     In this section, our main task is to prove Theorem~\ref{thm:DEF-Turan}.
     Recall that a collection of sets $A_1,\dots,A_k$ is called a ($k$-)sunflower (or $\Delta$-system) with core $C$ if $A_i\cap A_j=C$ for any different $i,j\in [k]$.

\subsection{The special case that $L=\{0,\ell\}$ with $r\neq 2\ell$}\label{subsec:special}
     The strategy of Helliar and Liu~\cite{helliar2024generalized} for proving Theorem~\ref{thm:Helliar-Liu-DEF-Turan} is to apply induction on both $r$ and the cardinality of the set $L$. Crucially, the key part of their proof lies in the base case that $L=\{0,\ell\}$, for which they used an interesting combination of the $\Delta$-system method and Tur\'an's theorem. To prove Theorem~\ref{thm:DEF-Turan} (1), we will also employ the induction strategy in spirit and first consider the base case.
     However, to prove our base case, we will replace the use of Tur\'an's theorem by a reductive step when combining the $\Delta$-system method.
     Our following result improves on Helliar and Liu's result for the base case. Note that here we require the sequence $0,\ell,r$  not to be an arithmetic progression.
 \begin{lemma}\label{lem:base}
     Let $r \ge 3$ and $0< \ell < r$ be fixed with $2\ell\neq r$. Then as $n$ goes to infinity we have
 	 \[
 	 \Psi_r(n,\{0,\ell\}) = o(n^2).
 	 \]
\end{lemma}

\begin{proof}
	 Let $G$ be an $n$-vertex $(K_{r},\{0,\ell\})$-intersecting graph.
	 For each subset $C\subset V(G)$ of size $\ell$, fix a maximum sunflower $\mathcal{S}_C\subset \mH_G^r$ with core $C$. As in \cite{helliar2024generalized}, let us define $U(G):=\{u\in V(G): d_{\mH_G^r}(u)\ge r^2\}$ and
	 \[
	 \C(G):=\left\lbrace C\in \binom{V(G)}{\ell} : |\mathcal{S}_C|\ge r^2 \right\rbrace.
	 \]
	 The following claim has been proved in \cite{helliar2024generalized}.
	 \begin{claim}[{\cite[Claims 5.2--5.4]{helliar2024generalized}}]\label{claim:S_C}
		If $G$ is a $(K_{r},\{0,\ell\})$-intersecting graph, then the following properties hold:
	 \begin{enumerate}[(a)]
		\item $U(G)=\bigcup_{C\in\C(G)} C$, 
		\item $C\cap C'=\emptyset$ for any distinct $C,C'\in\C(G)$, and
		\item for every $A\in \mH_G^r$ and every $C\in\C(G)$, we have either $A\in \mathcal{S}_C$ or $A\cap C=\emptyset$.
	\end{enumerate}
	\end{claim}
	Based on these properties, Helliar and Liu~\cite{helliar2024generalized} considered an auxiliary graph and then apply Tur\'an's theorem to obtain a reasonable upper bound. Here we first greedily delete vertices with small degree in the associated $r$-graph, and then apply Claim~\ref{claim:S_C} and use a reduction.
	
	Let us consider the following iterative process. Let $G_0:=G$. For $i\ge 0$, given the graph $G_i$ from the last step, if there exists some vertex $u$ in $G_i$ with $d_{\mH_{G_i}^r}(u)< r^2$, then delete $u$ from $G_i$ to obtain a new graph $G_{i+1}$; otherwise we stop. Suppose we finally stop at the graph $G_k$.  Clearly, we have
	\begin{equation}\label{eq:edges-G_k}
		|\mH_{G}^r|\le |\mH_{G_k}^r| + r^2 k \le |\mH_{G_k}^r| + r^2 n.
	\end{equation}
	If $V(G_k)=\emptyset$, \eqref{eq:edges-G_k} gives us $|\mH_{G}^r|=O(n)$ and we are done. Thus, we can assume now that $V(G_k)\neq\emptyset$. By \eqref{eq:edges-G_k}, it suffices to show that $|\mH_{G_k}^r|\le o(n^2)$ as $n$ increases.

Note that every vertex $u$ in $G_k$ satisfies $d_{\mH_{G_k}^r}(u)\ge r^2$ and hence $U(G_k)=V(G_k)$. Since $G_k$ preserves the $(K_{r},\{0,\ell\})$-intersecting property, by Claim~\ref{claim:S_C} (a) and (b), we know that  $\C(G_k)$ forms a partition of  $V(G_k)$ into $\ell$-subsets.
	Now consider an auxiliary graph $G'$ with vertex set $V(G')=\C(G_k)$ and edge set
	\[
	E(G') = \left\lbrace  \{C,C'\}\in \binom{\C(G_k)}{2} : C\cup C' \subset A \text{ for some } A\in \mH_{G_k}^r \right\rbrace.
	\]
It is clear that $|V(G')|=|\C(G_k)|=|V(G_k)|/\ell$.	The following claim indicates that $\ell\mid r$ when $V(G_k)\neq\emptyset$. Using this claim and the condition that $G_k$ is $(K_r,\{0,\ell\})$-intersecting, we will show that $G'$ is $(K_{r/\ell}, \{0,1\})$-intersecting and then apply the upper bound in \eqref{eq-Gowers-Janzer} to get our upper bound.
	
	\begin{claim}\label{claim:reduction1} 
		When $V(G_k)\neq\emptyset$, we have $\ell\mid r$ and $|\mH_{G_k}^r|= |\mH_{G'}^{r/\ell}|$.
	\end{claim}
	\begin{proof}
For any edge $A\in \mH_{G_k}^r$ and any $\ell$-set $C\in\C(G_k)$, by Claim~\ref{claim:S_C} (c), we have either $C\subset A$ or $C\cap A=\emptyset$. Since $A\subset V(G_k)=\bigcup_{C\in\C(G_k)} C$, we have that $A$ is a disjoint union of several $C\in\C(G_k)$. Thus $\ell\mid r$.

For every $A\in \mH_{G_k}^r$,  $\{C\in\C(G_k):C\subset A\}$  forms an $(r/\ell)$-clique in the auxiliary graph $G'$ by definition, which is an $(r/\ell)$-edge in $\mH_{G'}^{r/\ell}$ . Conversely, each $(r/\ell)$-clique  $\{C_1,\dots,C_{r/\ell}\}$ in $G'$ gives an $r$-clique $A=\bigcup_{i=1}^{r/\ell}C_i$ in $G_k$. Hence $|\mH_{G_k}^r|= |\mH_{G'}^{r/\ell}|$.
\end{proof}
		

    \begin{claim}\label{claim:reduction2}
	 The graph $G'$ is $(K_{r/\ell}, \{0,1\})$-intersecting.
    \end{claim}
    \begin{proof}
    	Let $B_1,B_2$ be any two different $(r/\ell)$-edges in $\mH_{G'}^{r/\ell}$. Let $A_i=\bigcup_{C\in B_i}C$, $i=1,2$, which are two $r$-edges in $\mH_{G_k}^r$.
%
     If $|B_1\cap B_2|\ge 2$, then $|A_1\cap A_2|\ge |\bigcup_{C\in B_1\cap B_2} C|\ge 2\ell$, which is impossible as $\mH_{G_k}^r$ is $\{0,\ell\}$-intersecting. Thus, the associated $(r/\ell)$-graph $\mH_{G'}^{r/\ell}$ of $G'$ is $\{0,1\}$-intersecting, that is, $G'$ is $(K_{r/\ell}, \{0,1\})$-intersecting.
    \end{proof}

    Now let us finish the proof of the lemma. By Claims~\ref{claim:reduction1} and \ref{claim:reduction2} we have \[
    |\mH_{G_k}^r|=|\mH_{G'}^{r/\ell}|\le \Psi_{r/\ell}(|V(G')|,\{0,1\}) = \Psi_{r/\ell}(|V(G_k)|/\ell,\{0,1\}) \le \Psi_{r/\ell}(n/\ell,\{0,1\}).
    \] Then by \eqref{eq:edges-G_k} we get
    \[
    |\mH_{G}^r|\le |\mH_{G_k}^r| + r^2 n
      \le \Psi_{r/\ell}(n/\ell,\{0,1\}) + r^2 n.
    \]
    Note that $r/\ell\ge 3$ since we assume $0<\ell<r$ and $2\ell\neq r$. Then it follows from \eqref{eq-Gowers-Janzer} that $|\mH_{G}^r|\le \Psi_{r/\ell}(n/\ell,\{0,1\}) + r^2 n \le o(n^2)$ as $n\to\infty$, which finishes the proof of Lemma~\ref{lem:base}.
\end{proof}

\subsection{Completing the proof of Theorem~\ref{thm:DEF-Turan} (1)}\label{subsec:general}
    For more general $L$ and $r$ with similar restrictions, one can still employ the induction on both $r$ and $|L|$ as in \cite{helliar2024generalized}. However, we find it somewhat more convenient here to establish a recursive relation (see Lemma~\ref{lem:relation}) which is a consequence of F\"uredi's fundamental structure theorem.

 \begin{theorem}[F\"uredi~\cite{furedi1983finite}]\label{thm:furedi}
	For $r\ge 2$ and $L\subset [0,r-1]$, there exists a positive constant $c = c(r)$ such that every $L$-intersecting $r$-graph $\F$ contains a subhypergraph $\F^*\subset \F$ with $|\F^*|\ge c|\F|$ satisfying all of the following properties.
	\begin{enumerate}[(i)]
		\item The families $\mathcal{I}(F):=\{F\cap F':F'\in\F^*\}$ are isomorphic for all $F\in\F^*$;
		\item For any $A\in \mathcal{I}(F)$ with $F\in\F^*$, $A$ is the core of an $(r+1)$-sunflower in $\F^*$;
		\item For any $F\in\F^*$, $\mathcal{I}(F)$ is closed under intersection, i.e., if $A_1,A_2\in \mathcal{I}(F)$ then $A_1\cap A_2\in \mathcal{I}(F)$;
		\item For any $A\in \mathcal{I}(F)$ with $F\in\F^*$, $|A|\in L$;
		\item For any different $A,A'\in \bigcup_{F\in\F^*} \mathcal{I}(F)$, we have $|A\cap A'|\in L$.
	\end{enumerate}
 \end{theorem}

 The following relation is a variant of \cite[Proposition 6.5]{deza1985sections}. Recall that $\Phi_r(n,L)$ denotes the maximum size of an $n$-vertex $L$-intersecting $r$-graph.

\begin{lemma}\label{lem:relation}
     For $L=\{\ell_1,\dots,\ell_s\}\subset [0,r-1]$ with $0\le\ell_1<\dots<\ell_s<r$ and any $i\in [s]$ we have \[
     \Psi_r(n, L) \le
       c^{-1}\max\{ \Phi_r(n, L\setminus\{\ell_i\}), \Phi_{\ell_i}(n, \{\ell_1,\dots,\ell_{i-1}\}) \Psi_{r-\ell_i}(n-\ell_i, \{0,\ell_{i+1}-\ell_i,\dots,\ell_s-\ell_i\}) \},
     \]
     where $c=c(r)$ is the constant from Theorem~\ref{thm:furedi}.
\end{lemma}
  \begin{proof}
  	Let $G$ be an $n$-vertex $(K_r,L)$-intersecting graph and let $\F=\mH_{G}^r$. Applying Theorem~\ref{thm:furedi} to $\F$ we can find a subhypergraph $\F^*\subset \F$ satisfying all of the properties in Theorem~\ref{thm:furedi}. If there is no $A\in \bigcup_{F\in\F^*} \mathcal{I}(F)$ such that $|A|=\ell_i$, then by the definition of $\mathcal{I}(F)$  we see that $\F^*$ is $L\setminus\{\ell_i\}$-intersecting, so $|\F|\le c^{-1}|\F^*|\le c^{-1}\Phi_r(n, L\setminus\{\ell_i\})$ and we are done.
  	
  	Hence we can assume now that for every $F\in\F^*$ there exists some $A\in\mathcal{I}(F)$ of cardinality $\ell_i$ by property (i).
  Let $\A_{F}:=\{A\in\mathcal{I}(F) : |A|=\ell_i\}$, which is nonempty for each $F\in\F^*$.
  	By property (v), it is clear that $\bigcup_{F\in\F^*}\A_{F}$ is an $\{\ell_1,\dots,\ell_{i-1}\}$-intersecting $\ell_i$-graph and hence
  	\begin{equation}\label{eq:l_i-sets}
  		\left| \bigcup_{F\in\F^*}\A_{F}\right| \le \Phi_{\ell_i}(n, \{\ell_1,\dots,\ell_{i-1}\}).
  	\end{equation}
  	For $A\in \bigcup_{F\in\F^*}\A_{F}$, let us denote $\F^*[A]:= \{F\in\F^* : A\subset F \}$ and $\F^*(A):= \{F\setminus A :  F\in\F^*[A]\}$.
  	
  	\begin{claim}\label{claim:containing-A}
  	  For any $A\in\binom{V(G)}{\ell_i}$, $|\F^*[A]|=|\F^*(A)|\le \Psi_{r-\ell_i}(n-\ell_i, \{0,\ell_{i+1}-\ell_i,\dots,\ell_s-\ell_i\})$.
  	\end{claim}
  	\begin{proof}
  	   For $v\in V(G)$, denote by $N_G(v)$ the set of vertices $w\in V(G)$ such that $w$ and $v$ are adjacent. Let $N=\bigcap_{v\in A} N_G(v)$. Since each $F\in \F^*[A]\subset \F=\mH_{G}^r$ induces an $r$-clique in $G$, $F\setminus A$ induces an $(r-\ell_i)$-clique in $G[N]$. So $\F^*(A)\subset \mH_{G[N]}^{r-\ell_i}$.  Note that the induced graph $G[N]$ is a $(K_{r-\ell_i},\{0,\ell_{i+1}-\ell_i,\dots,\ell_s-\ell_i\})$-intersecting graph on at most $n-\ell_i$ vertices. Therefore, 
  \[|\F^*[A]|=|\F^*(A)|\le |\mH_{G[N]}^{r-\ell_i}|\le \Psi_{r-\ell_i}(n-\ell_i, \{0,\ell_{i+1}-\ell_i,\dots,\ell_s-\ell_i\}).\qedhere\]
  	\end{proof}
  	Since any $F\in\F^*$ contains at least one $\ell_i$-set in $\A_{F}$, by Claim~\ref{claim:containing-A} and \eqref{eq:l_i-sets} we get
  	\begin{align*}
  		|\F^*|\le \left|\bigcup_{A\in\cup_{F\in\F^*}\A_{F}} \F^*[A] \right|
  		&\le \left| \bigcup_{F\in\F^*}\A_{F} \right| \cdot \Psi_{r-\ell_i}(n-\ell_i, \{0,\ell_{i+1}-\ell_i,\dots,\ell_s-\ell_i\}) \\
  		&\le \Phi_{\ell_i}(n, \{\ell_1,\dots,\ell_{i-1}\}) \Psi_{r-\ell_i}(n-\ell_i, \{0,\ell_{i+1}-\ell_i,\dots,\ell_s-\ell_i\}).
  	\end{align*}
  	Then the proof of Lemma~\ref{lem:relation} is completed by noting that $|\F|\le c^{-1}|\F^*|$.
  \end{proof}

  We also need the following lemma from the linear algebra method.
\begin{lemma}[{\cite[Ex. 7.3.8]{babai2020linear}}]\label{lem:mod}
	Let $q\ge 2$ be a prime power and let $L\subset [0,q-1]$ be a set of two integers. Suppose $\mathcal{H}$ is an $n$-vertex $r$-graph satisfying
	\begin{enumerate}[(1)]
		\item $r \notin L \pmod{q}$, and
		\item $|A\cap B| \in L \pmod{q}$ for any distinct $A,B\in \mathcal{H}$.
	\end{enumerate}
	Then $|\mathcal{H}|\le \binom{n}{2}$.
\end{lemma}

    When $|L|=3$, Lemma~\ref{lem:mod} can be used to get a bound in a special case:
\begin{lemma}\label{lem:|L|=3}
	 Let $0=\ell_1< \ell_2 <\ell_3 < r$ with $r-\ell_3=\ell_3-\ell_2 \neq\ell_2 $. Then any $n$-vertex $\{\ell_1,\ell_2,\ell_3\}$-intersecting $r$-graph $\mathcal{H}$ satisfies
	 \[
	 |\mathcal{H}|\le \binom{n}{2} =O(n^2).
	 \]
\end{lemma}

\begin{proof}
	  Since $\ell_3\neq 2\ell_2$ and $\ell_2< \ell_3$, we have $\ell_3\nmid 2\ell_2$. Then there must exist a prime power $q$ such that $q\mid\ell_3$ and $q\nmid 2\ell_2$. Hence
	 \begin{enumerate}[(i)]
	 	\item $r=2\ell_3-\ell_2 \equiv -\ell_2 \notin \{0,\ell_2\} \pmod{q}$, and
	 	\item $|A\cap B| \in \{0,\ell_2\} \pmod{q}$ for any distinct $A,B\in \mathcal{H}$.
	 \end{enumerate}
	 Then it follows from Lemma~\ref{lem:mod} that $|\mathcal{H}|\leq \binom{n}{2}$.
\end{proof}

    Now we finish the proof of Theorem~\ref{thm:DEF-Turan} (1) using Lemmas~\ref{lem:base}, \ref{lem:relation} and \ref{lem:|L|=3}.
   \begin{proof}[Proof of Theorem~\ref{thm:DEF-Turan} (1)]
   	Since $\ell_1,\ell_2,\dots,\ell_s,r$ do not form an arithmetic progression, then either $r-\ell_s\neq\ell_s-\ell_{s-1}$ or there exists $i\in [s-2]$ such that $ r-\ell_s=\ell_s-\ell_{s-1}=\dots=\ell_{i+2}-\ell_{i+1}\neq \ell_{i+1}-\ell_i$.
   	
   	\textbf{Case 1.} If $r-\ell_s\neq\ell_s-\ell_{s-1}$, then it follows from Lemma~\ref{lem:base} that $\Psi_{r-\ell_{s-1}}(n, \{0,\ell_{s}-\ell_{s-1}\})=o(n^2)$. Then by Lemma~\ref{lem:relation} and Theorem~\ref{thm:DEF78} we have
   	\begin{align*}
   		\Psi_r(n, L) &\le
   		c^{-1}\max\{ \Phi_r(n, L\setminus\{\ell_{s-1}\}), \Phi_{\ell_{s-1}}(n, \{\ell_1,\dots,\ell_{s-2}\}) \Psi_{r-\ell_{s-1}}(n-\ell_{s-1}, \{0,\ell_{s}-\ell_{s-1}\}) \} \\
   		&\le c^{-1}\max\left\lbrace  \prod_{\ell\in L\setminus\{\ell_{s-1}\}}\frac{n-\ell}{r-\ell},  \left( \prod_{\ell\in \{\ell_1,\dots,\ell_{s-2}\}}\frac{n-\ell}{\ell_{s-1}-\ell} \right) \Psi_{r-\ell_{s-1}}(n-\ell_{s-1}, \{0,\ell_{s}-\ell_{s-1}\}) \right\rbrace \\
   		& \le c^{-1}\max\{n^{s-1},n^{s-2}\cdot \Psi_{r-\ell_{s-1}}(n, \{0,\ell_{s}-\ell_{s-1}\})\} \le o(n^s).
   	\end{align*}
   	
   	\textbf{Case 2.} If there is an index $i\in [s-2]$ such that $ r-\ell_s=\ell_s-\ell_{s-1}=\dots=\ell_{i+2}-\ell_{i+1}\neq \ell_{i+1}-\ell_i$, then by Lemma~\ref{lem:|L|=3} we see that
   	\begin{equation}\label{eq:n-choose-2}
   		\Phi_{\ell_{i+3}-\ell_{i}}(n, \{0,\ell_{i+1}-\ell_{i},\ell_{i+2}-\ell_{i}\}) \le \binom{n}{2},
   	\end{equation}
   	where we set $\ell_{s+1}:=r$. One can then combine \eqref{eq:n-choose-2} with the proof of Lemma~\ref{lem:relation} to show that $\Phi_r(n,L)=O(n^{s-1})$ in this case. Since here our goal is to bound $\Psi_r(n,L)$ only, we can directly apply Lemma~\ref{lem:relation} as follows.
   	Applying Lemma~\ref{lem:relation} and Theorem~\ref{thm:DEF78} we obtain
   	\begin{align*}
   		\Psi_r(n, L) &\le
   		c^{-1}\max\{ \Phi_r(n, L\setminus\{\ell_{i}\}), \Phi_{\ell_{i}}(n, \{\ell_1,\dots,\ell_{i-1}\}) \Psi_{r-\ell_{i}}(n-\ell_{i}, \{0,\ell_{i+1}-\ell_{i},\dots,\ell_{s}-\ell_{i}\}) \} \\
   		&\le c^{-1}(n^{s-1}+n^{i-1}\Psi_{r-\ell_{i}}(n, \{0,\ell_{i+1}-\ell_{i},\dots,\ell_{s}-\ell_{i}\})).
   	\end{align*}
   	We claim that $\Psi_{r-\ell_{i}}(n, \{0,\ell_{i+1}-\ell_{i},\dots,\ell_{s}-\ell_{i}\}))=O(n^{s-i})$. Indeed, if $i=s-2$ then we are done by \eqref{eq:n-choose-2}. If $i\le s-3$, letting $L':=\{0,\ell_{i+1}-\ell_{i},\dots,\ell_{s}-\ell_{i}\}$ and using Lemma~\ref{lem:relation}, Theorem~\ref{thm:DEF78} and \eqref{eq:n-choose-2} we have
   	\begin{align*}
   		&\Psi_{r-\ell_{i}}(n, \{0,\ell_{i+1}-\ell_{i},\dots,\ell_{s}-\ell_{i}\})) \\
   		\le & c_i^{-1} \max\{ \Phi_r(n, L'\setminus\{\ell_{i+3}-\ell_{i}\}), \Phi_{\ell_{i+3}-\ell_{i}}(n, \{0,\ell_{i+1}-\ell_{i},\ell_{i+2}-\ell_{i}\}) \\ & ~~~~~~~~~~~~~~~~~~~~~~~~~~~~~~~~~~~~~~~~\cdot\Psi_{r-\ell_{i+3}}(n-(\ell_{i+3}-\ell_{i}), \{0,\ell_{i+4}-\ell_{i+3},\dots,\ell_{s}-\ell_{i+3}\}) \} \\
   		\le & c_i^{-1} \max\{n^{s-i}, \Phi_{\ell_{i+3}-\ell_{i}}(n, \{0,\ell_{i+1}-\ell_{i},\ell_{i+2}-\ell_{i}\})\cdot n^{s-i-2} \} \le O(n^{s-i}),
   	\end{align*}
   	where $c_i$ is a constant depending only on $r-\ell_i$, and the penultimate inequality is by Theorem~\ref{thm:DEF78}. Consequently, in this case we always have $\Psi_r(n, L)=O(n^{s-1})$. This completes the proof of Theorem~\ref{thm:DEF-Turan} (1).
   \end{proof}

\subsection{Proof of Theorem~\ref{thm:DEF-Turan} (2)}\label{subsec:last-case}
 In this subsection, we prove Theorem~\ref{thm:DEF-Turan} (2) by combining the following theorem due to Erd\H{o}s~\cite{erdos1962} and a result (see Claim~\ref{claim:FT}) of Frankl and Tokushige~\cite{frankl2016uniform} on the uniform eventown problem.
 \begin{theorem}[Erd\H{o}s~\cite{erdos1962}]\label{thm:erdos}
 	Let $n \ge t \ge r \ge 2$ be integers. Every $n$-vertex $K_{t+1}$-free graph $G$ satisfies
 	\[
 	N(K_r,G)\le N(K_r, T(n,t)),
 	\]
 	and equality holds if and only if $G$ is isomorphic to the Tur\'an graph $T(n,t)$.
 \end{theorem}

\begin{proof}[Proof of Theorem~\ref{thm:DEF-Turan} (2)]
	Suppose the sequence $\ell_1,\ell_2,\dots,\ell_s,r$ is an arithmetic progression and let $d>0$ be the common difference. Then $L=\{\ell_1,\ell_1+d,\dots,\ell_1+(s-1)d\}$ and  $r=\ell_1+sd$. Note that the case $d=1$ already follows from Theorem~\ref{thm:Helliar-Liu-EKR-Turan}, so assume now that $d\ge 2$. Write $n-\ell_1=md+\lambda$, where $\lambda$ is the unique integer in $[0,d-1]$ such that $d\mid (n-\ell_1-\lambda)$.

	Let us first give the extremal construction which is an ``atomic construction'' similar to the eventown theorem. Let $T(m,s)$ be the Tur\'an graph, and let $\hat{T}(m,s,d)$ be the blow-up of $T(m,s)$ by replacing each vertex in $T(m,s)$ by a clique $K_d$.
 That is, \[
	 \hat{T}(m,s,d):=  \underbrace{\lfloor m/s\rfloor K_d +\dots+ \lfloor m/s\rfloor K_d}_{s-s_1}
	  + \underbrace{\lceil m/s\rceil K_d +\dots+ \lceil m/s\rceil K_d}_{s_1},
	\]
where $s_1\in [0,s-1]$ is the integer satisfying $s\mid (m-s_1)$ and $bK_d$ denotes the  vertex-disjoint union of $b$  copies of $K_{d}$'s.
Finally, let \begin{equation}\label{gnrl}
               G_{n,r,L}:=K_{\ell_1} + \hat{T}(m,s,d).
             \end{equation}
Since $T(m,s)$ is $K_{s+1}$-free, $\hat{T}(m,s,d)$ is $K_{sd+1}$-free. Recall $r=sd+\ell_1$. Then every copy of $K_{r}$ in $G_{n,r,L}$ contains the fixed $K_{\ell_1}$ in its definition. So

 \[
	 N(K_{r}, G_{n,r,L})=N(K_{r-\ell_1}, \hat{T}(m,s,d))=N(K_{sd}, \hat{T}(m,s,d))=N(K_{s}, T(m,s)).
	 \]
Further, it is clear that $G_{n,r,L}$ is a $(K_r,L)$-intersecting graph with $\ell_1+md=n-\lambda$ vertices. Thus \[
	 \Psi_r(n,L)\ge N(K_{s}, T(m,s)) = N(K_{s}, T(\lfloor(n-\ell_1)/d\rfloor,s)) = (1+o(1))\left(\frac{n-\ell_1}{r-\ell_1} \right)^{s}.
	 \]

	
	 It remains to show that $\Psi_{r}(n,L)\le N(K_{s}, T(m,s))$ for large enough $n$. Since $\Psi_r(n,L)\ge (1+o(1))\left(\frac{n-\ell_1}{r-\ell_1} \right)^{s}$, by the second part of Theorem~\ref{thm:DEF78}, any largest associated $L$-intersecting $r$-graph satisfies that all $r$-edges have a common $\ell_1$-subset. So it suffices to show that $\Psi_{r-\ell_1}(n-\ell_1,L_{d,s}\})\le N(K_{s}, T(m,s))$ for large $n$, where $L_{d,s}:=\{0,d,2d,\dots,(s-1)d\}$.
	
 	 Let $G$ be an $(n-\ell_1)$-vertex $(K_{r-\ell_1},L_{d,s})$-intersecting graph with $N(K_{r-\ell_1}, G)=\Psi_{r-\ell_1}(n-\ell_1,L_{d,s})\ge N(K_{s}, T(m,s))$. Note that $r-\ell_1=sd$.
 	 Similar to the proof of Lemma~\ref{lem:base}, we will consider a collection of sets with properties in Claim~\ref{claim:S_C} (b) and (c). However, here such a collection will come from \textit{atoms} instead of cores of sunflowers.
 	 As in \cite{frankl2016uniform}, we say that a subset $S\subset V(G)$ with $|S|\ge d$ is an \textit{atom} if $S$ is inclusion maximal with the property that either $S\subset A$ or $S\cap A=\emptyset$ for any $A\in \mH_{G}^{r-\ell_1}=\mH_{G}^{sd}$. It is easy to see that atoms are pairwise disjoint, since the union of two non-disjoint sets satisfying the above property also satisfies this property. Let $\mathcal{S}\subset \binom{V(G)}{d}$ be the set consisting of all atoms of size $d$ and let $X_1:=\bigcup_{S\in \mathcal{S}} S\subset V(G)$ be the disjoint union of atoms in $\mathcal{S}$.

   Note that each atom $S\in \mathcal{S}$ plays a role that it is either included or excluded completely by an $sd$-clique in $G$. So we can view each $S$ as a whole to construct an auxiliary graph, say $H'$, whose  vertex set is $V(H')=\mathcal{S}$ and edge set is \[E(H')=\left\lbrace \{S_1,S_2\}\in\binom{\mathcal{S}}{2}:S_1\cup S_2\subset A \text{ for some } A\in \mH_{G}^{sd} \right\rbrace . \]
 	 By the definition of atoms, it is easy to see that $N(K_{sd}, G[X_1])=N(K_s, H')$.
 	 \begin{claim}\label{claim:K_{s+1}-free}
 	 	The graph $H'$ is $K_{s+1}$-free, and hence $N(K_{sd}, G[X_1])=N(K_s, H')\le N(K_s, T(|\mathcal{S}|,s))$.
 	 \end{claim}
 	 \begin{proof}
 	 	If there exists a copy of $K_{s+1}$ in $H'$, then we can find a copy of $K_{(s+1)d}$ in $G$ as vertices of $H'$ correspond to pairwise disjoint copies of $K_d$ in $G$. However, it is easy to see that there exist two copies of $K_{sd}$ in every copy of $K_{(s+1)d}$ such that the intersection size of their vertex sets is $sd-1$, contradicting to the assumption that $G$ is $(K_{sd},L_{d,s})$-intersecting. Thus we conclude that $H'$ is $K_{s+1}$-free. Then it follows from Theorem~\ref{thm:erdos} that $N(K_s, H')\le N(K_s, T(|\mathcal{S}|,s))$.
 	 \end{proof}
 	
 	Let $X_0:=V(G)\setminus X_1$. 
 The following claim shows that for every $x\in X_0$, the number of $sd$-cliques in $G$ containing $x$ is very small. The claim has been proved by Frankl and Tokushige~\cite{frankl2016uniform} in the setting of set systems. For the reader's convenience, we include its proof sketch below.

 	
 	 \begin{claim}[{\cite[Claim 5]{frankl2016uniform}}]\label{claim:FT}
   For any $x\in X_0$, we have $|\{A: x\in A\in \mH_{G}^{sd}\}| =O(n^{s-2})$.
 	 \end{claim}
 	
 	 \begin{proof}[Proof sketch]
 	 Note that $\mH_{G}^{sd}(x):=\{A\setminus\{x\}:A\in \mH_{G}^{sd}, x\in A\}$ is $\{d-1,2d-1,\dots,(s-1)d-1\}$-intersecting. If $I:=\bigcap_{B\in \mH_{G}^{sd}(x)} B$ has size at most $d-2$, then we are done by the second part of Theorem~\ref{thm:DEF78}. Moreover, if $|I|\ge d$, then $\mH_{G}^{sd}(x)$ is now $\{2d-1,\dots,(s-1)d-1\}$-intersecting and we are also done by Theorem~\ref{thm:DEF78}. Hence assume $|I|=d-1$. Observe that $\{x\}\cup I$ is not contained in an atom since any atom containing $\{x\}\cup I$ must be $\{x\}\cup I$ itself and $x\notin X_1$. Thus, there exists $D\in \mH_{G}^{sd}$ such that $x\notin D$ and $I\cap D\neq\emptyset$. Clearly, any $B\in \mH_{G}^{sd}(x)$ intersects $D\setminus I$. Hence $|\mH_{G}^{sd}(x)|\le \sum_{y\in D\setminus I} |\{B\in \mH_{G}^{sd}(x):I\cup\{y\}\subset B\}|\le O(n^{s-2})$ by Theorem~\ref{thm:DEF78}.
 	\end{proof}


 	 Recall that $|V(G)|=n-\ell_1=md+\lambda$.
     Using the above two claims we can show that $|X_0|=\lambda$ and $|X_1|=md$.
 	 \begin{claim}\label{claim:X_0=lambda}
 	 	For sufficiently large $n$, we have $|X_0|=\lambda$ and $|X_1|=md$. 
 	 \end{claim}
 	 \begin{proof}
 	 Note that any $A\in \mH_{G}^{sd}\setminus \mH_{G[X_1]}^{sd}$ contains some vertex $x\in X_0$. Applying Claims~\ref{claim:K_{s+1}-free} and \ref{claim:FT} we then have
 	 \begin{align*}
 	 	|\mH_{G}^{sd}| &= |\mH_{G}^{sd}\setminus \mH_{G[X_1]}^{sd}| + |\mH_{G[X_1]}^{sd}| \\
 	 	&\le \sum_{x\in X_0}|\mH_{G}^{sd}(x)| + N(K_{sd}, G[X_1]) \\
 	 	&\le |X_0| \cdot O(n^{s-2}) + N(K_s, T(|\mathcal{S}|,s)).
 	 \end{align*}
 	 	
 	 Since $d\mid |X_1|$ and $n-\ell_1=md+\lambda$, we have $|X_0|=n-\ell_1-|X_1| \equiv \lambda \pmod{d}$. We prove by contradiction. Assume $|X_0|\neq \lambda$. Then $|X_0|\ge d+\lambda$ and hence $|\mathcal{S}|=(n-\ell_1-|X_0|)/d\le m-1$. Let $m'\equiv m\pmod s$ and let $m''\equiv m-1 \pmod s$ where $m',m''\in [0,s-1]$. Then for sufficiently large $n$ we have
 	 \begin{align*}
 	 	&N(K_s, T(m,s)) - N(K_s, T(|\mathcal{S}|,s)) \\
 	 	\ge & N(K_s, T(m,s)) - N(K_s, T(m-1,s))\\
 	 	= & \left(\frac{m-m'}{s} \right)^{s-m'} \left(\frac{m+s-m'}{s} \right)^{m'}- \left(\frac{m-1-m''}{s} \right)^{s-m''} \left(\frac{m-1+s-m''}{s} \right)^{m''}\\
 	 	= & s^{-s}( -m'(s-m')+(s-m')m' + (1+m'')(s-m'')+(1-s+m'')m'')m^{s-1} - O(m^{s-2})\\
 	 	= & (sd)^{-s+1} n^{s-1} - O(n^{s-2})
 	 	=  \Omega(n^{s-1}) > |X_0| \cdot O(n^{s-2}),
 	 \end{align*}
 	 where the last inequality holds when $|X_0|< cn$ for a sufficiently small constant $c\in (0,1)$ (depending only on $d,s$).
 	 On the other hand, when $|X_0|\ge cn$ and $n$ is sufficiently large,
 	 \begin{align*}
 	 	&N(K_s, T(m,s)) - N(K_s, T(|\mathcal{S}|,s)) \\
 	 	\ge & N(K_s, T(m,s)) - N(K_s, T(\lceil((1-c)n-\ell_1)/d\rceil,s))\\
 	 	= & (1+o(1))\left(\frac{n-\ell_1-\lambda}{sd} \right)^s-(1+o(1))\left(\frac{(1-c)n-\ell_1}{sd} \right)^s\\
 	 	=&\Omega(n^s) > |X_0| \cdot O(n^{s-2}).
 	 \end{align*}
 	 Thus, we have
 	 	\[
 	 	|\mH_{G}^{sd}|\le |X_0| \cdot O(n^{s-2}) + N(K_s, T(|\mathcal{S}|,s)) < N(K_{s}, T(m,s)), 
 	 	\]
 	 contradicting to the assumption that $G$ satisfies $N(K_{sd}, G)\ge N(K_{s}, T(m,s))$.
 	 	
 	 Therefore, when $n$ is large enough, we get $|X_0|=\lambda<d$ and $|X_1|=md$.
 	 \end{proof}

 Next, we claim that in fact for each $x\in X_0$, $|\{A: x\in A\in \mH_{G}^{sd}\}| =0$. This is true since  each $A\in\mH_{G}^{sd}$ satisfies $|A\cap X_0|=sd-|A\cap X_1|\equiv 0\pmod{d}$, which is $0$ by $|X_0|=\lambda<d$ and definitions of atoms.  Thus $\mH_{G}^{sd}\subset \mH_{G[X_1]}^{sd}$ and  $N(K_{sd}, G)= N(K_{sd}, G[X_1])$.

 	 Now we finish the proof of Theorem~\ref{thm:DEF-Turan} (2).
 	 Combining Claims~\ref{claim:K_{s+1}-free} and \ref{claim:X_0=lambda}, we have $|\mathcal{S}|=m$ and
 	 \[
 	  \Psi_{r-\ell_1}(n-\ell_1,L_{d,s}\})
 	  =  N(K_{sd}, G)
 	  = N(K_{sd}, G[X_1])
 	  \le N(K_s, T(|\mathcal{S}|,s))=N(K_{s}, T(m,s)),
 	 \]
 	 whenever $n$ is sufficiently large.
 	 Consequently, $\Psi_{r}(n,L)= N(K_{s}, T(m,s))$ for large enough $n$, and it follows directly from the proof that $G_{n,r,L}$ is the unique extremal structure.
\end{proof}

     Let us conclude this section with a discussion of a problem proposed by Helliar and Liu~\cite{helliar2024generalized}.
     \begin{problem}[\cite{helliar2024generalized}]
        Characterize the family of sets $L \subset [0,r-1]$ such that
        \begin{equation}\label{eq-lim=0}
        \lim_{n\to\infty} \frac{\Psi_r(n,L)}{\Phi_r(n,L)} = 0.
        \end{equation}
     \end{problem}
     \noindent Denote by $\mathcal{L}_r$ the family of sets $L \subset [0,r-1]$ such that \eqref{eq-lim=0} holds.
     For $L=[0,t]$ with $t\in [r-2]$, a seminal result of R\"odl~\cite{Rodl} says $\Phi_r(n,L)\ge (1-o(1))\binom{n}{t+1}/\binom{r}{t+1}=\Omega(n^{|L|})$, while \eqref{eq-Gowers-Janzer} shows $\Psi_r(n,L)=o(n^{t+1})=o(n^{|L|})$; hence $[0,t]\in \mathcal{L}_r$ for any $t\in [r-2]$.
     More generally, if the sequence $\ell_1<\dots<\ell_{s}<r$ is not an arithmetic progression for some $s\ge 2$, then Theorem~\ref{thm:DEF-Turan} (1) indicates that $\{\ell_1,\dots,\ell_{s}\} \in \mathcal{L}_r$ whenever $\Phi_r(n,\{\ell_1,\dots,\ell_{s}\})=\Omega(n^{s})$.  For instance, it follows straightforwardly from R\"odl's result that $\Phi_{\ell+kd}(n,\{\ell,\ell+d,\dots,\ell+(s-1)d\})=\Omega(n^{s})$ for any $\ell\ge 0, d\ge 1$ and $k>s$. As another example, it was proven by Tokushige~\cite{Tokushige2006} that $\Phi_{12}(n,\{0,1,2,3,4,6\})=\Omega(n^{6})$ (see \cite{Tokushige2006} for more such examples). On the other hand, Theorem~\ref{thm:DEF-Turan} (2) shows that $\{\ell_1,\dots,\ell_{s}\} \notin \mathcal{L}_r$ when the sequence $\ell_1<\dots<\ell_{s}<r$ is an arithmetic progression.

\section{Hilton--Milner-type stability for $(K_r,t)$-intersecting graphs}\label{sec:hm}
     In this section, we prove Theorem~\ref{thm:stab} by using the method in \cite{frankl1978intersecting} and some new ideas. The main difference from \cite{frankl1978intersecting} is that we need to rule out one of the extremal cases corresponding to \eqref{eq:1st-extremal}. We will also make use of the following result which is a special case of \cite[Theorem 14]{gerbner2019generalized}.
     \begin{theorem}[\cite{gerbner2019generalized}]\label{thm:GMV}
     	If $G$ is an $n$-vertex $(K_{s+1},1)$-intersecting graph, then \[N(K_s,G)\le (1+o(1))\left(\frac{n}{s} \right) ^s.\]
     \end{theorem}

     \begin{proof}[Proof of Theorem~\ref{thm:stab}]
     Let $G$ be an $n$-vertex non-trivial $(K_r,t)$-intersecting graph with $N(K_r, G)$ maximized. Then $N(K_r, G) \ge N(K_r,  K_{t+2} + T(n-t-2, r-t-1))$.  For $j\in [t,r]$, let
     \[
     \T_j := \left\lbrace T\in \binom{V(G)}{j} : T\subset A \text{ for some } A\in \mH_G^r, \text{ and } |T\cap B|\ge t \text{ for any } B\in \mH_G^r \right\rbrace.
     \]
     Note that $\T_t=\emptyset$ since $\mH_G^r$ is non-trivial $t$-intersecting. For $j\in [t+1,r]$, define
     \[
     \T_j' := \left\lbrace T\in \T_j : T \text{ is inclusion minimal in } \bigcup_{k\ge t} \T_k, \text{ i.e., there is no } S\in \bigcup_{t\le k<j} \T_k \text{ such that } S\subset T \right\rbrace.
     \]
      Since $A\in \T_r$ for any $A\in \mH_G^r$, we see that each $A\in \mH_G^r$ contains some member of $\bigcup_{j>t} \T_j'$ by taking a minimal subset of $A$ in $\bigcup_{k>t} \T_k$. Thus, $\mH_G^r\subset \bigcup_{j>t} \bigcup_{T\in \T_j'} \{A\in \mH_G^r : T\subset A \}$ and hence
     \begin{equation}\label{eq:count-by-base0}
     	|\mH_G^r| \le \sum_{j>t} \sum_{T\in \T_j'} |\{A\in \mH_G^r : T\subset A \}|.
     \end{equation}
      We will show that $|\mH_G^r|$ and $\sum_{T\in \T_{t+1}'} |\{A\in \mH_G^r : T\subset A \}|$ differ only by a small error term.
     It is not difficult to verify the following claim by noting that the core of such an $(r-t+2)$-sunflower also belongs to $\bigcup_{k>t} \T_k$ and using the minimality in the definition of $\T_j'$.
     \begin{claim}[{\cite[Lemma 1]{frankl1978intersecting}}]
     	For any $j>t$, the family $\T_j'$ contains no $(r-t+2)$-sunflower.
     \end{claim}

     The well-known Erd\H{o}s--Rado sunflower lemma~\cite{erdos1960intersection} then implies that $|\T_j'|$ is bounded by a function depending only on $r$. Hence $\sum_{T\in \bigcup_{j\ge t+2}\T_j'} |\{A\in \mH_G^r : T\subset A \}|\le |\bigcup_{j\ge t+2}\T_j'|\cdot n^{r-t-2} = O(n^{r-t-2})$.
     Define
     \[
     \T_{t+1}'' := \{T\in \T_{t+1}' : T \text{ is contained in more than } rn^{r-t-2} \text{ edges of } \mH_G^r\}.
     \]
     It follows that $\sum_{T\in \T_{t+1}'\setminus\T_{t+1}''} |\{A\in \mH_G^r : T\subset A \}|\le |\T_{t+1}'\setminus\T_{t+1}''|\cdot rn^{r-t-2}=O(n^{r-t-2})$.
     Combining this with \eqref{eq:count-by-base0}, we obtain
     \begin{equation}\label{eq:count-by-base}
     	|\mH_G^r| \le \sum_{j>t} \sum_{T\in \T_j'} |\{A\in \mH_G^r : T\subset A \}|
     	\le \sum_{T\in \T_{t+1}''} |\{A\in \mH_G^r : T\subset A \}| +O(n^{r-t-2}).
     \end{equation}
     The following claim gives a suitable bound on $|\{A\in \mH_G^r : T\subset A \}|$ for each $T\in \T_{t+1}''$.

     \begin{claim}\label{claim:containing-t+1}
     	Let $n$ be sufficiently large. For any subset $T\subset V(G)$ of size $t+1$ which induces a $(t+1)$-clique in $G$ (or simply $T\in \T_{t+1}''$), we have \[
     	|\{A\in \mH_G^r : T\subset A \}| \le  (1+o(1))\left(\frac{n-t-1}{r-t-1} \right)^{r-t-1}. \]
     \end{claim}
     \begin{proof}
     	We first show that $G[N]$ is $(K_{r-t},1)$-intersecting, where $N:=\bigcap_{v\in T} N_G(v)\subset V(G)\setminus T$.
     	Indeed, for any two different $(r-t)$-sets $A,B\in \mH_{G[N]}^{r-t}$ and any different $x,y\in T$ we have $A\cup T\setminus\{x\}, B\cup T\setminus\{y\} \in\mH_{G}^{r}$, whose intersection has size $|A\cap B|+t-1$, so $|A\cap B|\ge 1$.
     	Thus, $G[N]$ is $(K_{r-t},1)$-intersecting. It follows from Theorem~\ref{thm:GMV} that
     	\[
     	|\{A\in \mH_G^r : T\subset A \}|=|N(K_{r-t-1},G[N])| \le  (1+o(1))\left(\frac{n-t-1}{r-t-1} \right)^{r-t-1}. \qedhere
     	\]
     \end{proof}

      Combining \eqref{eq:count-by-base} with Claim~\ref{claim:containing-t+1}, we see that
      \begin{equation}\label{eq:count-base-2}
      	|\mH_G^r| \le (1+o(1))\left(\frac{n-t-1}{r-t-1} \right)^{r-t-1} |\T_{t+1}''|+O(n^{r-t-2}).
      \end{equation}

     \begin{claim}
      	When $n$ is large, we have $\T_{t+1}''=\binom{D}{t+1}$ for some $(t+2)$-set $D\subset V(G)$.
     \end{claim}
     \begin{proof}
     	Recall that by the maximality of $N(K_r, G)$ we have
     	\begin{equation}\label{eq:max-assumption3}
     		N(K_r, G) \ge N(K_r,  K_{t+2} + T(n-t-2, r-t-1))= (1+o(1))(t+2)\left(\frac{n-t-2}{r-t-1} \right)^{r-t-1}.
     	\end{equation}
     	Combining this with \eqref{eq:count-base-2} gives $|\T_{t+1}''|\ge t+2\ge 3$. Let $T_1,T_2,T_3\in \T_{t+1}''$ be three different sets in $\T_{t+1}''$.
     	
     	Next, we show that $\T_{t+1}''$ is $t$-intersecting. For any two different $T,T'\in \T_{t+1}''$, if $|T\cap T'|<t$, then any set in $\{A\in\mH_G^r:T\subset A\}$ intersects $T'\setminus T$ and hence \[|\{A\in\mH_G^r:T\subset A\}|\le\sum_{x\in T'\setminus T} |\{A\in\mH_G^r:T\cup\{x\}\subset A\}|\le r\binom{n-t-2}{r-t-2}\le rn^{r-t-2},
     	\] which contradicts to the definition of $\T_{t+1}''$. Thus we see that $\T_{t+1}''$ is $t$-intersecting.
     	
        Since $\T_{t+1}''$ is $t$-intersecting, $|T_1\cap T_2|=|T_2\cap T_3|=|T_1\cap T_3|=t$. Let us show $T_1\cap T_2\neq T_1\cap T_3$ by contradiction. Assume on the contrary that $T_1\cap T_2=T_1\cap T_3 \triangleq C$, and write $T_i=\{x_i\}\cup C$ for $i\in [3]$. Then for all $T\in\T_{t+1}''$ we have $C\subset T$, since otherwise $|T\cap C|=t-1$ and $x_1,x_2,x_3\in T\setminus C$ which would imply $|T|=|T\setminus C|+|T\cap C|\ge t+2$. Thus in this case $\T_{t+1}''$ is a sunflower with core $C$. Note that there exists some $A_0\in \mH_G^r$ such that $C\not\subset A_0$ as $\mH_G^r$ is non-trivial $t$-intersecting. Furthermore, the definition of $\T_{t+1}''$ implies that $|A_0\cap T|\ge t$ for any $T\in \T_{t+1}''$; hence $|A_0\cap C|=t-1$ and $P\triangleq(\bigcup_{T\in \T_{t+1}''} T)\setminus C \subset A_0$.
        
        Let $M:=V(G)\setminus(C\cup A_0)$, and for $x\in A_0\setminus C$ let \[
        \mH_x:=\{B\subset \mH_{G[M]}^{r-t-1}:\{x\}\cup C\cup B\in \mH_G^r\} \subset \mH_{G[N_0]}^{r-t-1},
        \]
        where $N_0:=M\cap (\bigcap_{u\in C} N_G(u))$. See \autoref{fig1} for a visual diagram of these set notations.

        \begin{center}
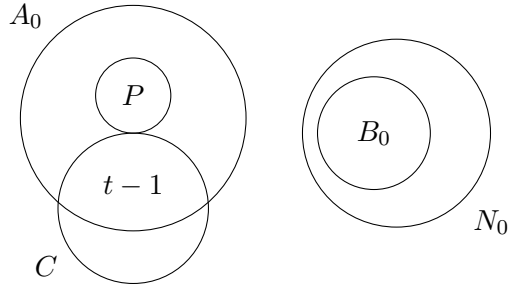

        \begin{tikzpicture}
        	
        	\node [draw,
        	circle,
        	minimum size =3cm,
        	label={135:$A_0$}] (A) at (0,1.2){};
        	
        	\node [draw,
        	circle,
        	minimum size =1cm
        ]
            (T) at (0,1.5){};
        	
        	\node [draw,
        	circle,
        	minimum size =2cm,
        	label={210:$C$}] (C) at (0,0){};
        	
        	\node [draw,
        	circle,
        	minimum size =2.5cm,
        	label={315:$N_0$}] (N) at (3.5,1){};
        	
        	\node [draw,
        	circle,
        	minimum size =1.5cm,
        	] (B) at (3.2,1){};
        	
        	\node at (0,0.3) {$t-1$};
        	
        	\node at (T.center){$P$};
        	
        	\node at (B.center){$B_0$};
        \end{tikzpicture} \end{center} \captionof{figure}{A visual diagram of set notations.}\label{fig1}

        \noindent We claim that $\mH_x\cap\mH_{x'}=\emptyset$ for any different $x,x'\in A_0\setminus C$. Indeed, if there exists $B_0\in \mH_x\cap\mH_{x'}$, then $\{x,x'\}\cup (C\setminus \{v\})\cup B_0 \in \mH_G^r$ for any $v\in C$. However, since $|\T_{t+1}''|\ge 3$, we can find some $y\in P\setminus\{x,x'\}$; hence we get $\{y\}\cup C\in \T_{t+1}''$ and 
        \[
        |(\{y\}\cup C)\cap (\{x,x'\}\cup (C\setminus \{v\})\cup B_0)|=|C\setminus\{v\}|<t,
        \] contradicting to the definition of $\T_{t+1}''$. This shows that $\mH_x\cap\mH_{x'}=\emptyset$ for any different $x,x'\in A_0\setminus C$. 
        
        Moreover, observe that $G[N_0]$ is $K_{r-t}$-free, since otherwise the union of $C$ and a copy of $K_{r-t}$ in $G[N_0]$ would give an $r$-clique disjoint with $A_0\setminus C$ which leads to a contradiction as $|A_0\cap C|=t-1$. Thus, $\sum_{x\in A_0\setminus C} |\mH_x|\le|\mH_{G[N_0]}^{r-t-1}|\le N(K_{r-t-1},T(|N_0|,r-t-1))$ by Theorem~\ref{thm:erdos}. By the definition of $\T_{t+1}''$, we have $|A\cap T|\ge t$ for any $A\in\mH_G^r$ and any $T\in \T_{t+1}''$, so
        \begin{align*}
        	|\mH_G^r|\le &|\{A\in \mH_G^r: C\subset A \text{ and } A\cap (A_0\setminus C)\neq\emptyset\}| +|\{A\in \mH_G^r: |A\cap C|=t-1 \text{ and } P \subset A \}|\\
        	\le &\sum_{x\in A_0\setminus C} |\mH_x| + O(n^{r-t-2})
        	\le N(K_{r-t-1},T(|N_0|,r-t-1))+ O(n^{r-t-2}),
        \end{align*}
        contradicting to \eqref{eq:max-assumption3} for large $n$. Therefore, $T_1\cap T_2\neq T_1\cap T_3$.

        Write $T_1\setminus T_2=\{y_1\}$ and $T_2\setminus T_1=\{y_2\}$, and let $D=T_1\cup T_2$. Clearly, $|D|=t+2$. Since $\T_{t+1}''$ is $t$-intersecting and $T_1\cap T_2\neq T_1\cap T_3$, we have $y_1,y_2\in T_3$, $|T_1\cap T_2\cap T_3|=t-1$ and $T_3\subset D$. Furthermore, since for $T\in\T_{t+1}''\setminus\{T_1,T_2,T_3\}$ either $T_1\cap T\neq T_1\cap T_2$ or $T_1\cap T\neq T_1\cap T_3$, we similarly have $T\subset D$ for any $T\in\T_{t+1}''$ by using the fact that $\T_{t+1}''$ is $t$-intersecting. On the other hand, \eqref{eq:count-base-2} and \eqref{eq:max-assumption3} imply $|\T_{t+1}''|\ge t+2 =\binom{|D|}{t+1}$. Thus we conclude that $\T_{t+1}''=\binom{D}{t+1}$.
     \end{proof}
        Since $\T_{t+1}''=\binom{D}{t+1}$ for large $n$, by the definition of $\T_{t+1}''$ we see that $|A\cap D|\ge t+1$ for any $A\in\mH_G^r$. Additionally, it is clear that $G[D]\cong K_{t+2}$ as each $T\in\T_{t+1}''$ is contained in some edge in $\mH_G^r$. Write $\T_{t+1}''=\binom{D}{t+1}=\{T_1,\dots,T_{t+2}\}$. For $i\in [t+2]$ let $N^i:=(V(G)\setminus D)\cap \bigcap_{v\in T_i} N_G(v)$ and let $N^0:=\bigcap_{v\in D} N_G(v) = \bigcap_{i\in [t+2]} N^i$.
        Observe that for any $A\in\mH_G^r$, it follows from $|A\cap D|\ge t+1$ that either $A\cap D=T_i$ and $A\setminus T_i\subset N^i$ for some $i\in [t+2]$, or $D\subset A$ and $A\setminus D\subset N^0$. This yields
        \[
        N(K_r,G)\le N(K_{r-t-2},G[N^0]) + \sum_{i=1}^{t+2} N(K_{r-t-1},G[N^i]).
        \]

        Note that for each $i\in [t+2]$, $G[N^i]$ is $K_{r-t}$-free, since otherwise the union of $T_i$ and a copy of $K_{r-t}$ in $G[N^i]$ would give an edge $A\in\mH_G^r$ with $|A\cap D|=t< t+1$, which is impossible. Thus, by Theorem~\ref{thm:erdos}, we get
        \[
        N(K_{r-t-1},G[N^i]) \le N(K_{r-t-1},T(|N^i|,r-t-1))
        \]
        for each $i\in [t+2]$. Similarly,
        \[
        N(K_{r-t-2},G[N^0]) \le N(K_{r-t-2},T(|N^0|,r-t-1)).
        \]
        Therefore,
        \begin{align*}
        	N(K_r,G)
        	&\le N(K_{r-t-2},T(|N^0|,r-t-1)) + \sum_{i=1}^{t+2} N(K_{r-t-1},T(|N^i|,r-t-1))\\
        	&\le N(K_{r-t-2},T(n-t-2,r-t-1)) + (t+2) N(K_{r-t-1},T(n-t-2,r-t-1))\\
            & = N(K_r,K_{t+2}+T(n-t-2,r-t-1)).
        \end{align*}
       Moreover, when the equality holds, $|N^i|=n-t-2$ for $i\in [0,t+2]$; hence $N^0=N^1=\dots=N^{t+2}=V(G)\setminus D$ and $G=G[D]+G[V(G)\setminus D]$. Combining this with the second part of Theorem~\ref{thm:erdos} shows that the equality holds if and only if $G$ is isomorphic to $K_{t+2} + T(n-t-2, r-t-1)$.
     \end{proof}

 \section{Concluding remarks}\label{sec:conclusion}
 In this work, we first showed that, for $L=\{\ell_1,\dots,\ell_s\}\subset [0,r-1]$ with $\ell_1<\dots<\ell_s$ and $s\notin\{1,r\}$, as $n$ goes to infinity $\Psi_r(n,L)=\Theta_{r,L}(n^{|L|})$ if and only if the sequence $\ell_1,\dots,\ell_s,r$ is an arithmetic progression. Moreover, in this case we determined the exact values of $\Psi_r(n,L)$ for sufficiently large $n$, thereby providing a generalized Tur\'an extension of the uniform eventown problem~\cite{frankl2016uniform}. These improved a recent result of Helliar and Liu~\cite{helliar2024generalized}. Additionally, for the generalized Tur\'an extension of the Erd\H{o}s--Ko--Rado theorem due to Helliar and Liu~\cite{helliar2024generalized}, we showed a Hilton--Milner-type stability result. As remarked in \cite{helliar2024generalized}, there are many classical theorems from extremal set theory that one could consider extending to the generalized Tur\'an setting. Below we list some questions regarding potential generalized Tur\'an extensions.

 \subsection{Cliques with the cover-free property}
 It might be of interest to consider generalized Tur\'an extensions of packing-like structures, e.g., that of cover-free families. An $r$-graph $\F$ is said to be {\it $t$-cover-free} if in $\F$ no edge is contained in the union of $t$ others. It is well known \cite{erdos1985families,frankl1987colored} that the maximum size of a $t$-cover-free $r$-graph $\F\subset \binom{[n]}{r}$ is $\Theta_{r,t}(n^{\lceil r/t\rceil})$ as $n$ grows. For the corresponding generalized Tur\'an problem, define
 \[
 C_r(n,t):= \max\{ N(K_r, G):   G \text{ is an $n$-vertex graph such that $\mH_G^r$ is $t$-cover-free}\}.
 \]
 It is easy to see that \eqref{eq-Gowers-Janzer} implies $C_r(n,t)> n^{\lceil r/t\rceil-o(1)}$, where $o(1)\to 0$ as $n\to\infty$.
 \begin{question}
 	Let $r,t\ge 2$ be fixed integers. Is it true that $C_r(n,t)=o(n^{\lceil r/t\rceil})$ as $n\to\infty$?
 \end{question}
 \noindent Using the hypergraph removal lemma, we are able to show that  it is true when $r\equiv 1 \pmod t$.

 \subsection{Stability of the Erd\H{o}s matching problem in the generalized Tur\'an setting}
 The renowned Erd\H{o}s matching problem~\cite{erdos1965problem} asks for the maximum size of an $r$-graph $\mH\subset \binom{[n]}{r}$ without $k+1$ pairwise disjoint edges. In the generalized Tur\'an setting, Zhang, Chen, Gy\H{o}ri and Zhu~\cite{zhang2024maximum} recently showed that, if the associated $r$-graph $\mH_G^r$ of an $n$-vertex graph $G$ contains no $k+1$ pairwise disjoint edges, then
 \[
 N(K_r,G)\le N(K_r,K_k + T(n-k,r-1))
 \]
 for fixed $r,k$ and sufficiently large $n$. Moreover, $K_k + T(n-k,r-1)$ is the unique extremal graph. Note that when $k=1$ this is a special case of Theorem~\ref{thm:Helliar-Liu-EKR-Turan} and we have the corresponding stability result. On the other hand, Bollob\'as, Daykin and Erd\H{o}s~\cite{BDE76} proved a stability result for the Erd\H{o}s matching problem. So a natural question would be to give a stability result of the generalized Tur\'an extension of the Erd\H{o}s matching problem.

\end{document}